\documentclass[twoside,a4paper, 12pt]{article} 

\usepackage[left=3cm, right=3cm, top=3cm, bottom=3cm]{geometry}

\usepackage[T1]{fontenc}
\usepackage[utf8]{inputenc}
\usepackage[french,english]{babel}

\usepackage{amsmath,amssymb,amsfonts,amsthm}
\usepackage{calrsfs}

\usepackage[charter]{mathdesign}
\DeclareMathAlphabet{\matcal}{OMS}{zplm}{m}{n}

\usepackage[pdftex]{graphicx}	
\usepackage{caption}
\usepackage{url}

\usepackage[colorlinks=true]{hyperref}
\hypersetup{colorlinks,
            citecolor=[rgb]{0.4,0,0},  
            linkcolor=[rgb]{0.05,0.20,0.55}}

\newcommand{\rouge}[1]{{\color{red}#1}}

\usepackage{bbm}

\usepackage{sectsty}		
\sectionfont{\normalfont } 

\subsectionfont{\normalfont \bfseries }

\subsubsectionfont{\normalfont \bfseries }

\usepackage{enumitem}
\newlist{hyp0}{enumerate}{1}
\setlist[hyp0]{label=(H\arabic*)}

\usepackage[numbers,comma,sort]{natbib}
\usepackage{amsmath}	

\numberwithin{equation}{section}

\newtheorem{theorem}{Theorem}[section]

\newtheorem*{theorem*}{Theorem}
\newtheorem{lemma}[theorem]{Lemma}

\newtheorem{proposition}[theorem]{Proposition}

\usepackage[affil-it]{authblk}
\usepackage{fancyhdr}
\fancypagestyle{empty}{\fancyfoot[C]{\vspace*{-1.0\baselineskip}\thepage}}

\usepackage[titletoc]{appendix}
\AtBeginDocument{}
\appendixtocon

\begin{document}

\newcommand{\D}{\mathbb{D}}
\newcommand{\R}{\mathbb{R}}
\newcommand{\Q}{\mathbb{Q}}
\newcommand{\Z}{\mathbb{Z}}
\newcommand{\N}{\mathbb{N}}
\newcommand{\EE}{\mathbb{E}}
\newcommand{\PP}{\mathbb{P}}
\newcommand{\dd}{\textrm{d}}
\newcommand{\Fcal}{\mathcal{F}}
\newcommand{\Hcal}{\matcal{H}}
\newcommand{\HH}{\matcal{H}_H}
\newcommand{\vbar}{\bar{v}_{\xi,H}}
\newcommand{\vH}{v_{\xi,H}}
\newcommand{\e}{\mathbf{e}_n}
\newcommand{\ee}{\mathbf{e}_{2n}}

\newcommand{\checked}{\rouge{[ $\tfrac{1}{4}<H<\tfrac{1}{2}$ checked] }}
\newcommand{\si}{s_{(i)}}
\newcommand{\sii}{s_{(ii)}}
\newcommand{\siii}{s_{(iii)}}
\newcommand{\Cb}{\matcal{C}_b}
\newcommand{\sbullet}{{\scalebox{0.4}{$\bullet$}}}

\title{\normalfont \Large  {Noise sensitivity of functionals of fractional Brownian motion 
driven stochastic differential equations: Results and perspectives}}
\vspace{-1cm}

\author[1]{Alexandre Richard}
\author[2]{Denis Talay}
\affil[1]{\small CMAP, École Polytechnique, Route de Saclay, 91128 Palaiseau, France}
\affil[2]{\small INRIA Sophia-Antipolis,  2004 route des Lucioles, F-06902 Sophia-Antipolis, France}

\date{\small \today}

\maketitle

\begin{abstract}
We present an innovating sensitivity analysis for stochastic 
differential equations:
We study the sensitivity, when the Hurst 
parameter~$H$ of the driving fractional Brownian motion tends to the 
pure Brownian value, of 
probability distributions of smooth functionals of the trajectories of 
the solutions 
$\{X^H_t\}_{t\in \mathbb{R}_+}$ and of the Laplace transform of 
the first passage time of $X^H$ at a given threshold.
Our technique requires to extend already known Gaussian 
estimates on the density of $X^H_t$ to estimates with constants
which are uniform w.r.t. $t$ in in the whole half-line $\R_+-\{0\}$ and 
$H$ when $H$ tends to~$\tfrac{1}{2}$.
\end{abstract}

{\sl Key words\/}: Fractional Brownian motion, Malliavin calculus, first hitting time.

\section{Introduction}
\label{sec:Intro}

Recent statistical studies show memory effects in biological, 
financial, physical data: see e.g.~\cite{Rypdal} for a statistical 
evidence in climatology and \cite{ComteCoutinRenault} for a financial model and citations therein for evidence in finance.
For such data the Markov structure of L\'evy driven stochastic 
differential equations makes such models questionable. It seems worth
proposing new models driven by noises with 
long-range memory such as fractional Brownian motions.

In practice the accurate estimation of the Hurst parameter~$H$ of the 
noise is difficult (see e.g.~\cite{BerzinEtAl}) and therefore one needs 
to 
develop 
sensitivity analysis w.r.t.~$H$ of probability distributions of smooth 
and non smooth functionals of the solutions $(X^H_t)$ to stochastic 
differential equations. Similar ideas were developed in \cite{JolisViles} for symmetric integrals of the fractional Brownian motion.

Here we review and illustrate by numerical experiments our theoretical 
results obtained in~\cite{RichardTalay} for two extreme situations in 
terms of Malliavin regularity: on the one hand, expectations of smooth 
functions of the solution at a fixed time; on the other hand,
Laplace transforms of first passage times at 
prescribed thresholds.
Our motivation to consider first passage times comes from their many 
use in various applications:
default risk in mathematical finance or spike trains in neuroscience
(spike trains are sequences of times at which the membrane potential of neurons reach 
limit thresholds and then are reset to a resting value, 
are essential to describe the neuronal activity),
stochastic numerics
(see e.g. \cite[Sec.3]{bossy-al}) and physics
(see e.g.~\cite{M-R-O-14}). 
Long-range dependence leads to analytical and numerical 
difficulties: see e.g. \cite{Jeon-al}.

Our theoretical estimates and numerical results tend to show that the 
Markov Brownian model is a good proxy 
model as long as the Hurst parameter remains close to~$\frac{1}{2}$.
This robustness property, even for probability distributions of 
singular functionals (in the sense of Malliavin calculus) of the paths 
such as first hitting times, is an important information for modeling 
and simulation purposes:
when statistical or calibration procedures lead to estimated values 
of~$H$ close to $\frac{1}{2}$, then it is reasonable to work with  
Brownian SDEs, which allows to analyze the model by means of PDE 
techniques and
stochastic calculus for semimartingales, and to simulate it by means of 
standard stochastic simulation methods. 

\subsubsection*{Our main results} 
The fractional Brownian motion $\{B^H_t\}_{t\in \R_+}$ with Hurst 
parameter 
$H\in(0,1)$ 
is the centred Gaussian process with covariance
\begin{equation*}
R_H(s,t) = \tfrac{1}{2}\left(s^{2H} + t^{2H} - |t-s|^{2H}\right), \quad 
\forall s,t\in \R_+.
\end{equation*}
Given $H\in(\tfrac{1}{2},1)$, we consider the
process $\{X^H_t\}_{t\in \R_+}$ solution to the following stochastic 
differential equation driven by~$\{B^H_t\}_{t\in \R_+}$:
\begin{equation}\label{eq:fSDE_0}
X^H_t = x_0 + \int_0^t b(X^H_s)\ \dd s + \int_0^t \sigma(X^H_s) \circ 
\dd B^H_s , \tag{1;H}
\end{equation}
where the last integral is a pathwise Stieltjes integral in the sense 
of \cite{Young}. For $H=\tfrac{1}{2}$ the process $X$ 
solves the following SDE in the classical Stratonovich sense:
\begin{equation}\label{eq:SDE_0}
X_t = x_0 + \int_0^t b(X_s)\ \dd s + \int_0^t \sigma(X_s) \circ \dd B_s 
.\tag{1;$\tfrac{1}{2}$}
\end{equation}

Below we use the following set of hypotheses:
\begin{hyp0}
\item\label{hyp:h1} There exists $\gamma\in (0,1)$ such that 
$b,\sigma\in \matcal{C}^{1+\gamma}(\R)$;
\item\label{hyp:h1'} $b,\sigma \in \matcal{C}^2(\R)$;
\item\label{hyp:h2} The function $\sigma$ satisfies a strong 
ellipticity condition: $\exists \sigma_0>0$ such that $|\sigma(x)|\geq 
\sigma_0, 
\forall x\in \R$.
\end{hyp0}

Our first theorem is elementary. It describes the sensitivity w.r.t. 
$H$ around the 
critical Brownian parameter $H=\tfrac{1}{2}$ of time marginal 
probability distributions of $\{X^H_t\}_{t\in \R_+}$.

\begin{theorem}\label{th:gapLawXt^H}
Let $H\in(\tfrac{1}{2},1)$, and let $X^H$ and $X$ be as before. 
Suppose that $b$ and $\sigma$ satisfy \ref{hyp:h1} and \ref{hyp:h2}, 
and $\varphi$ is 
bounded and H\"older continuous of order $2+\beta$ for some $\beta>0$. 
Then, for any $T>0$ there exists $C_T>0$ such that 
\begin{equation*}
\forall H\in[\tfrac{1}{2},1),~~
\sup_{t\in [0,T]} \left|\EE \varphi(X_t^H) - \EE \varphi(X_t) \right| 
\leq C_T\ (H-\tfrac{1}{2}).
\end{equation*}
\end{theorem}

Our next theorem concerns the first passage time at 
threshold~1 of $X^H$ issued from $x_0<1$: $\tau^X_H := \inf\{t\geq 0: 
X^H_t=1\}$. 
The probability distribution of the first passage time $\tau_H$ of a 
fractional Brownian motion is not explictly known. \cite{Molchan} 
obtained the asymptotic behaviour of its tail 
distribution function and \cite{DecreusefondNualart08} obtained an 
upper bound on the 
Laplace transform of $\tau_H^{2H}$. 
The recent work of \cite{DelormeWiese} proposes an asymptotic 
expansion (in terms of $H-\tfrac{1}{2}$) of the density of $\tau_H$
formally obtained by perturbation analysis techniques.

\begin{theorem}\label{prop:majgap_0}
Suppose that $b$ and $\sigma$ satisfy Hypotheses \ref{hyp:h1'} and 
\ref{hyp:h2}
and let $x_0<1$. 
There exist constants $\lambda_0\geq 1$, $\mu\geq 0$ (both depending on 
$b$ and $\sigma$ only), $\alpha>0$ and $0<\eta_0<\tfrac{1-x_0}{2}$ such that: for 
all $\epsilon\in (0,\tfrac{1}{4})$ and $0<\eta\leq \eta_0$, there exists 
 $C_{\epsilon,\eta}>0$ such that
\begin{align*}
\forall \lambda\geq \lambda_0,\ \forall H\in[\tfrac{1}{2},1), \quad &\left|\EE\left(e^{-\lambda \tau^X_H}\right) - \EE\left(e^{-\lambda 
\tau^X_{\frac{1}{2}}}\right)\right| \\
& \quad\quad\quad\quad\leq C_{\epsilon,\eta} 
(H-\tfrac{1}{2})^{\frac{1}{2}-\epsilon}\ 
e^{-\alpha S(1-x_0-2\eta) (\sqrt{2\lambda + \mu^2} - \mu)},
\end{align*}
where $S(x) = x\wedge x^{\frac{1}{2H}}$. In the pure fBm case (where 
$b\equiv 
0$ and $\sigma\equiv 1$) the result holds with $\lambda_0=1$ and 
$\mu=0$.
\end{theorem}

To prove the preceding theorem we need accurate estimates
on the density of $X^H_t$ with constants which are uniform w.r.t. small 
and long times and w.r.t. $H$ in $[\tfrac{1}{2},1)$. Our next theorem
improves estimates in~\cite{BaudoinEtAl,BKHT}. Our contributions 
consists in getting constants
which are uniform w.r.t. $t$ in the whole half-line $\R_+-\{0\}$ and 
$H$ when 
$H$ tends to~$\tfrac{1}{2}$.

\begin{theorem}\label{th:densityXH}
Assume that $b$ and $\sigma$ satisfy the conditions \ref{hyp:h1'} and 
\ref{hyp:h2}. Then for every $H\in[\tfrac{1}{2},1)$, the density of 
$X^H$ satisfies: there exists 
$C(b,\sigma)\equiv C>0$ such that, for all $t\in \R_+$ and $H\in[\tfrac{1}{2},1)$,
\begin{equation}\label{eq:boundDensXH}
\forall x\in \R,~~p_t^H(x) \leq \frac{e^{C t}}{\sqrt{2\pi t^{2H}}} 
~\exp\left(-\frac{(x-x_0)^2}{2 \|\sigma\|_\infty^2 t^{2H}}\right)  .
\end{equation}
\end{theorem}

Note that Theorems \ref{th:gapLawXt^H}, \ref{prop:majgap_0} and \ref{th:densityXH} are proved in \cite{RichardTalay}, including extensions to $H\in(\tfrac{1}{3},\tfrac{1}{2})$. We do not address the proof of Theorem \ref{th:densityXH} here.\\
We sketch the proofs of Theorems \ref{th:gapLawXt^H} and 
\ref{prop:majgap_0} in Section~\ref{sec:ErrorDec}. 
In Section~\ref{sec:conj} we consider a case which was 
not tackled in~\cite{RichardTalay}, that is, the case $\lambda<1$. 
Finally, in Section \ref{sec:numerics} we show numerical experiment 
results which illustrate Theorem~\ref{prop:majgap_0} and suggest that 
the $(H-\tfrac{1}{2})^{\frac{1}{2}-}$ rate is sub-optimal.

\section{Sketch of the proofs}\label{sec:ErrorDec}

\subsection{Reminders on Malliavin calculus}

We denote by $D$ and $\delta$ the classical derivative and Skorokhod 
operators of Malliavin calculus w.r.t. Brownian motion on the time 
interval~$[0,T]$ (see e.g. 
\cite{Nualart}). In the fractional Brownian motion framework
the Malliavin derivative $D^H$
is defined as an operator on the smooth random variables with values in
the Hilbert space $\matcal{H}_H$ defined as the completion of the space of step functions on $[0,T]$ with the following scalar product:
\begin{equation*}
\langle \varphi, \psi\rangle_{\matcal{H}_H} := \alpha_H \int_0^T \int_0^T 
\varphi_s\ \psi_t\ |s-t|^{2H-2}\ \dd s\dd t <\infty,
\end{equation*}
where $\alpha_H = H(2H-1)$.\\
The domain of $D^H$ in $L^p(\Omega)$ ($p>1$) is denoted by 
$\mathbb{D}^{1,p}$ and is the closure of the space of smooth random 
variables with respect to the norm:
\begin{equation*}
\|F\|_{1,p}^p = \EE(|F|^p) + \EE\left(\|D^H 
F\|_{\matcal{H}_H}^p\right).
\end{equation*}
Equivalently, $D^H$ and $\delta_H$ are defined as $D^H := (K_H^*)^{-1} 
D$ and $\delta_H(u) := \delta(K_H^* u)$ for 
$u \in (K_H^*)^{-1}(\text{dom} \delta)$ (cf.~\cite[p.288]{Nualart}),
where for any $H\in (\tfrac{1}{2},1)$ the operator $K_H^*$ is defined  
as follows: for any $\varphi$ with suitable integrability properties,
\begin{align*}
K_H^*\varphi(s) = (H-\tfrac{1}{2}) c_H\int_s^T 
\left(\frac{\theta}{s}\right)^{H-\frac{1}{2}} 
(\theta-s)^{H-\frac{3}{2}}\ \varphi(\theta)\ \dd \theta
\end{align*}
with
$$ c_H := \left( \frac{2H\ \Gamma(3/2-H)}{\Gamma(H+\tfrac{1}{2})\ 
\Gamma(2-2H)}\right)^{\tfrac{1}{2}}.$$

We denote by $\|\cdot\|_{\infty,[0,T]}$ the sup norm and $\|\cdot\|_\alpha$ the Hölder norm for functions on the interval $[0,T]$. Under Assumption \ref{hyp:h2}, there exists a transformation $F$ called the Lamperti transform, such that $X^H$ is mapped to the solution of ~(\ref{eq:fSDE_0}) with coefficients $\tilde{b} = \frac{b\circ F^{-1}}{\sigma\circ F^{-1}}$ and $\sigma\equiv 1$. Since $F$ is one-to-one, we assume in the rest of this paper that $\sigma$ is uniformly $1$. See \cite{RichardTalay} for details on the Lamperti transform in this framework.\\
Let $X^H$ be the solution to~(\ref{eq:fSDE_0}). There exist modifications of the processes $X^H$ and $D^H_\cdot X^H_\cdot$ such that for any $\alpha<H$ it a.s. holds that
\begin{equation} \label{ineq:DXH}
\begin{cases}
& \|X^H\|_{\infty,[0,T]} \leq C_T(1+|x_0| + \|B^H\|_{\infty,[0,T]}), \\
& \|X^H\|_\alpha \leq \|B^H\|_\alpha + C_T
(1+|x_0|+\|B^H\|_{\infty,[0,T]}), \\
& \|D^H_\cdot X^H_\cdot\|_{\infty,[0,T]^2} \leq C_T\ ,\\
& \sup_{r\leq t} \frac{|D^H_r X^H_t - 1|}{t-r} \leq C_T\ , \forall t\in 
[0,T]\ .
\end{cases}
\end{equation}
These inequalities are simple consequences of the definition of $X^H$, assumptions \ref{hyp:h1} and \ref{hyp:h2}, and the equality: $D^H_r X^H_t = \mathbf{1}_{\{r\leq t\}}\left(1+\int_r^t D^H_r X^H_s b'(X^H_s) ds\right)$ (see Section 3 in \cite{RichardTalay} for more details).

\subsection{Sketch of the proof of Theorem~\ref{th:gapLawXt^H}}

Proving Theorem~\ref{th:gapLawXt^H} is easy. A first technique consists 
in using
pathwise estimates on $B^H-B^{1/2}$ with $B^H$ and $B^{1/2}$ defined on the 
same 
probability space. A second technique, which we present here in order 
to introduce the reader to the method of proof for 
Theorem~\ref{prop:majgap_0}, consists in differentiating $u(t,X^H_t)$ 
where 
$$ u(s,x) := \EE_x\left(\varphi(X_{t-s})\right), $$
which leads to
\begin{align*}
u(t,X_t^H) &= u(0, x_0) + \int_0^t \left(\partial_s u(s,X_s^H) + 
\partial_x 
u(s,X_s^H) b(X_s^H)\right)\ \dd s + \delta_H\left(\mathbf{1}_{[0,t]} 
\partial_x u(\cdot,X_\cdot^H)  \right) \\
& \quad\quad + \alpha_H \int_0^t \int_0^s |r-s|^{2H-2} D^H_r X_s^H 
~\partial^2_{xx} u(s,X_s^H)~\dd r \dd s.
\end{align*}
As $u$ solves a parabolic PDE driven by the generator of $(X_t)$  and 
as the Skorokhod integral has zero mean we get
\begin{equation*}
\begin{split}
\EE \varphi(X_t^H) - \EE_{x_0} \varphi(X_t) &= \EE u(t,X_t^H) - u(0, 
x_0)\\
&= \EE\int_0^t \partial^2_{xx} u(s,X_s^H) \left( Hs^{2H-1} 
-\tfrac{1}{2}\right)
\dd s \\
&\quad +\alpha_H \EE\int_0^t \int_0^s |r-s|^{2H-2} (D^H_r X_s^H-1) 
\partial^2_{xx} u(s,X_s^H)\ \dd r \dd s.
\end{split}
\end{equation*}
It then remains to use the estimates~(\ref{ineq:DXH}).

\subsection{Sketch of the proof of Theorem~\ref{prop:majgap_0}}\label{subsec:proof_Thm}

We now sketch the proof of Theorem~\ref{prop:majgap_0}.
We will soon limit ourselves to the pure fBm case ($b(x)\equiv0$ and 
$\sigma\equiv 1$) in order to show the main ideas used in the proof and 
avoid too many technicalities. For now, our previous remark on the Lamperti transform implies that $\sigma$ can be chosen uniformly equal to $1$.

Our Laplace transforms sensititivity analysis is based on a PDE 
representation of first hitting time Laplace transforms in the case
$H=\tfrac{1}{2}$.

For $\lambda>0$ it is well known that
$$ \forall x_0\in(-\infty,1],~
\EE_{x_0}\left(e^{-\lambda \tau_{\frac{1}{2}}}\right)
=u_\lambda(x_0), $$
where the function $u_\lambda$ is the classical solution with bounded 
continuous first and second derivatives to
\begin{equation}\label{eq:ODE_Y-BM}
\begin{cases}
2b(x)u_\lambda'(x) + u_\lambda''(x) = 2 \lambda u_\lambda(x),~x<1, \\
u_\lambda(1) = 1, \\
\lim_{x\rightarrow -\infty} u_\lambda(x) = 0.
\end{cases}
\end{equation}

For any $t\in[0,T]$ the process $\mathbf{1}_{[0,t]} 
u_\lambda'(B^H_\sbullet)\ e^{-\lambda \sbullet}$ is in $\textrm{dom}\ 
\delta_H^{(T)}$. 
One thus can apply It\^o's formula to~$e^{-\lambda t} u_\lambda(X^H_{t})$
(see \cite[Section 2]{RichardTalay} and \cite{Nualart}). As $u_\lambda$ 
satisfies~(\ref{eq:ODE_Y-BM}), for any $t\leq T\wedge \tau_H$ we 
get
\begin{align*}
e^{-\lambda t} u_\lambda(X_t^H)  &=u_\lambda(x_0)+ \int_0^t e^{-\lambda s} \left(u_\lambda'(X_s^H) \tilde{b}(X_s^H) - \lambda u_\lambda(X_s^H)\right) \ \dd s + \delta^{(T)}_H \left( \mathbf{1}_{[0,t]}(\sbullet) e^{-\lambda \sbullet} u_\lambda'(X_\sbullet^H)\right) \\
& \quad \quad +\alpha_H \int_0^t \int_0^t D_v^H \left( e^{-\lambda s} u_\lambda'(X_s^H)\right) |s-v|^{2H-2} \ \dd v \dd s\ ,
\end{align*}
where the last term corresponds to the Itô term. 
Using $D_v^H X_s^H = \mathbf{1}_{[0,s]}(v)\left(1+ \int_0^s b'(X_\theta^H)\ D^H_v X_\theta^H \ \dd \theta\right)$ and the ODE (\ref{eq:ODE_Y-BM}) satisfied by $u_\lambda$, we get
\begin{align*}
e^{-\lambda t} u_\lambda(X_t^H)  &=u_\lambda(x_0) +\int_0^t \left(\alpha_H \int_0^s|s-v|^{2H-2} \dd v -\frac{1}{2}\right) e^{-\lambda s} u_\lambda''(X_s^H)\ \dd s  \\
&\quad\quad +\delta^{(T)}_H \left( \mathbf{1}_{[0,t]}(\sbullet) e^{-\lambda \sbullet} u_\lambda'(X_\sbullet^H)\right)  \\
&\quad\quad + \alpha_H \int_0^t \int_0^s e^{-\lambda s}  w_\lambda''(X_s^H)\ I(v,s)\ |s-v|^{2H-2} \ \dd v \dd s ,
\end{align*}
where $I(v,s) = \mathbf{1}_{\{v\leq s\}} \int_v^s b'(X_\theta^H)\ D^H_v X_\theta^H \ \dd \theta$. Observe that the last term vanishes for $H$ close to $\tfrac{1}{2}$, since $\alpha_H |s-v|^{2H-2}$ is an approximation of the identity and $I(v,s)$ converges to $0$ as $|v-s|\rightarrow 0$. This argument is made rigorous in \cite{RichardTalay}.

We now limit ourselves to the pure fBm case ($b(x)\equiv0$ and 
$\sigma\equiv 1$) to make the rest of the computations more understandable, although the differences will be essentially technical. Given that now, $u_\lambda'(x) = \sqrt{2\lambda} u_\lambda(x)$, the previous equality becomes
\begin{align*}
u_\lambda(B^H_{t})\ e^{-\lambda t}
&= u_\lambda(x_0) + \sqrt{2\lambda}\delta_H^{(T)}\left(\mathbf{1}_{[0,t]} 
u_\lambda(B^H_\sbullet)\ e^{-\lambda \sbullet}\right)+2\lambda \int_0^{t} 
\left(H s^{2H-1}-\tfrac{1}{2}\right)\ u_\lambda(B^H_s)\ e^{-\lambda s}\ \dd s.
\end{align*}
Evaluate the previous equation at $T\wedge \tau_H$, take 
expectations and let $T$ tend to infinity. For any $\lambda\geq 0$ it 
comes:
\begin{align}\label{eq:gapLaplace}
\EE\left(e^{-\lambda \tau_H}\right) - \EE\left(e^{-\lambda 
\tau_{\frac{1}{2}}}\right) &= \EE \left[ 2\lambda \int_0^{\tau_H} (H 
s^{2H-1}-\tfrac{1}{2}) u_\lambda(B^H_s)\ e^{-\lambda s}\ \dd s \right] \\
&\quad\quad+\sqrt{2\lambda}\lim_{T\rightarrow 
\infty}\EE\left[\left.\delta_H^{(T)}\left(\mathbf{1}_{[0,t]}
u_\lambda(B^H_\sbullet)\
 e^{-\lambda \sbullet}\right)\right|_{t=\tau_H\wedge T}\right]\nonumber \\
&=: I_1(\lambda) + I_2(\lambda).
\end{align}

\begin{proposition}\label{prop:I1}
Let $T$ be the function of $\lambda\in \R_+$ defined by $T(\lambda) = 
(2\lambda)^{1-\frac{1}{4H}}$ if $\lambda\leq 1$ and $T(\lambda) = 
\sqrt{2\lambda}$ if $\lambda>1$.
There exists a constant $C>0$ such that
$$ |I_1(\lambda)| \leq C\ (H-\tfrac{1}{2})\ 
e^{-\frac{1}{4} S(1-x_0) T(\lambda)} ,$$
where $S$ is the function defined in Theorem \ref{prop:majgap_0}.
\end{proposition}

\begin{proof}[Sketch of proof]
From Fubini's theorem, we get
\begin{equation*}
I_1(\lambda) = 2\lambda \int_0^{+\infty} (H 
s^{2H-1}-\tfrac{1}{2}) \EE \left[ \mathbf{1}_{\{\tau_H\geq s\}}u_\lambda(B^H_s)\right]\ e^{-\lambda s}\ \dd s 
\end{equation*}
The inequalities
\begin{equation*}
\forall H\in(\tfrac{1}{2},1),~\forall s\in (0,\infty),~|Hs^{2H-1} -\tfrac{1}{2}| 
\leq (H-\tfrac{1}{2})\ (1 \vee s^{2H-1}) |1+2H \log s|
\end{equation*}
and 
\begin{equation*}
\EE \left[ \mathbf{1}_{\{\tau_H\geq s\}}u_\lambda(B^H_s)\right] \leq \int_{-\infty}^1 u_\lambda(x) \frac{e^{-\frac{x^2}{2s^{2H}}}}{\sqrt{2\pi s^{2H}}}  \dd x = \int_{-\infty}^1 e^{-(1-x)\sqrt{2\lambda}} \frac{e^{-\frac{x^2}{2s^{2H}}}}{\sqrt{2\pi s^{2H}}}  \dd x
\end{equation*}
lead to the desired result. 
\end{proof}
Note that this proof adapts to diffusions, but that the density of $X^H$ is now needed, which is the purpose of Theorem \ref{th:densityXH}.

Compared to the proof of Theorem~\ref{th:gapLawXt^H}, an important 
difficulty appears when estimating~$|I_2(\lambda)|$: as the optional 
stopping theorem does not hold for 
Skorokhod integrals of the fBm one has to carefully estimate 
expectations of stopped 
Skorokhod integrals and obtain estimates which decrease infinitely fast 
when $\lambda$ goes to infinity.
We obtained the following result.

\begin{proposition}
\begin{equation}\label{eq:resultI2}
\forall\lambda>1,~|I_2(\lambda)| \leq C 
(H-\tfrac{1}{2})^{\frac{1}{2}-\epsilon} 
e^{-\alpha S(1-x_0-2\eta) \sqrt{2\lambda}} .
\end{equation}
\end{proposition}

\begin{proof}
Proposition 13 of~\cite{PeccatiThieullenTudor} shows that
$$\forall T>0,\quad 
{\EE\left(\left.\delta^{(T)}
(\mathbf{1}_{[0,t]}(\sbullet)
u_\lambda(B^H_\sbullet) 
e^{-\lambda\sbullet})\right|_{t=T\wedge \tau_H} \right) = 0}.$$ 
Thus $I_2(\lambda)$ satisfies
\begin{align*}
|I_2(\lambda)| &= \sqrt{2\lambda}\ \left|\lim_{N\rightarrow \infty}
\EE\left[\left.\delta_H^{(N)}\left(
\mathbf{1}_{[0,t]}(\sbullet)
u_\lambda(B^H_\sbullet) 
e^{-\lambda\sbullet}
\right)\right|_{t=\tau_H\wedge N} - 
\left.\delta^{(N)}\left(
\mathbf{1}_{[0,t]}(\sbullet)
u_\lambda(B^H_\sbullet) 
e^{-\lambda\sbullet} \right)\right|_{t=\tau_H\wedge N}\right] \right|\\
&=\sqrt{2\lambda} \left|\lim_{N\rightarrow \infty}
\EE\left[\left.\delta^{(N)}\left(\{K_H^*-\text{Id}\}(
\mathbf{1}_{[0,t]}(\sbullet)
u_\lambda(B^H_\sbullet) 
e^{-\lambda\sbullet})\right)\right|_{t=\tau_H\wedge N}\right] \right|\\
&\leq \sqrt{2\lambda} \lim_{N\rightarrow \infty} 
\EE\sup_{t\in [0,\tau_H\wedge N]}
|\delta^{(N)}\left(\{K_H^*-\text{Id}\}(
\mathbf{1}_{[0,t]}(\sbullet)
u_\lambda(B^H_\sbullet) 
e^{-\lambda\sbullet})\right)| \\
&\leq \sqrt{2\lambda} \lim_{N\rightarrow \infty} 
\EE\sup_{t\in [0,N]} \left[\mathbf{1}_{\{\tau_H\geq t\}} 
|\delta^{(N)}\left(\{K_H^*-\text{Id}\}(
\mathbf{1}_{[0,t]}(\sbullet)
u_\lambda(B^H_\sbullet) 
e^{-\lambda\sbullet})\right)|\right].
\end{align*}
Define the field $\{U_t(v),t\in[0,N], v\geq0\}$ 
and the process $\{\Upsilon_t,t\in[0,N]\}$ by
\begin{equation*}
\forall t\in[0,N],~
U_t(v) = \{K_H^* - \textrm{Id}\}\left( 
\mathbf{1}_{[0,t]}(\sbullet)\ 
u_\lambda(B_\sbullet^H)\ e^{-\lambda \sbullet} \right)(v),
\end{equation*}
and
\begin{equation*}
\Upsilon_t = \delta^{(N)}(U_t(\sbullet)).
\end{equation*}
For any real-valued function $f$ with $f(0)=0$ one has
\begin{equation*}
\begin{split}
\mathbf{1}_{\{\tau_H\geq t\}} |f(t)| &\leq 
\mathbf{1}_{\{\tau_H\geq t\}}\sum_{n=0}^{[t]} 
\sup_{s\in[n,n+1]} \mathbf{1}_{\{\tau_H\geq s\}}|f(s)-f(n)| \\
&\leq 
\sum_{n=0}^{[t]} 
\sup_{s\in[n,n+1]} \mathbf{1}_{\{\tau_H\geq s\}}|f(s)-f(n)|.
\end{split}
\end{equation*}
Therefore
\begin{equation}\label{eq:1stBoundI2}
\begin{split}
|I_2(\lambda)| 
&\leq \sqrt{2\lambda} \lim_{N\rightarrow \infty} 
\EE\sup_{t\in [0,N]} \left[\mathbf{1}_{\{\tau_H\geq t\}} |\Upsilon_t|
\right] \\
&\leq \sqrt{2\lambda} \lim_{N\rightarrow \infty} \sum_{n=0}^{N-1}
\EE\sup_{t\in [n,n+1]}\left[\mathbf{1}_{\{\tau_H\geq t\}}| \Upsilon_t - 
\Upsilon_n|\right].
\end{split}
\end{equation}

Suppose for a while that we have proven: there exists 
$\eta_0\in(0,\tfrac{1-x_0}{2})$ such that for all $\eta\in (0,\eta_0]$ 
and all $\epsilon\in (0,\tfrac{1}{4})$, there exist constants 
$C,\alpha>0$ such that
\begin{equation} \label{prop:supY}
\EE\sup_{t\in [n,n+1]}\left[ \mathbf{1}_{\{\tau_H\geq t\}} | 
\Upsilon_t- \Upsilon_n|\right] 
\leq C\ (H-\tfrac{1}{2})^{\frac{1}{2}-\epsilon}\ 
e^{-\frac{1}{3(2+4\epsilon)} \lambda n}
e^{-\alpha S(1-x_0-2\eta) \sqrt{2\lambda}}.
\end{equation}
We would then get:
\begin{align*}
|I_2(\lambda)| &\leq C\ \sqrt{2\lambda} \sum_{n=0}^\infty 
e^{-\frac{\lambda n}{3(2+4\epsilon)}} 
(H-\tfrac{1}{2})^{\frac{1}{4}-\epsilon} e^{-\alpha S(1-x_0-2\eta) 
\sqrt{2\lambda}} \\
&\leq C\ (H-\tfrac{1}{2})^{\frac{1}{2}-\epsilon} e^{-\alpha 
S(1-x_0-2\eta) \sqrt{2\lambda}},
\end{align*}
which is the desired result (\ref{eq:resultI2}).

In order to estimate the left-hand side of Inequality~(\ref{prop:supY})
we aim to apply Garsia-Rodemich-Rumsey's lemma (see below). However, it 
seems hard to get the desired estimate  
by estimating moments of increments of 
$\mathbf{1}_{\{\tau_H\geq t\}}|\Upsilon_t - \Upsilon_n|$,
in particular because
$\mathbf{1}_{\{\tau_H\geq t\}}$ is not smooth in the Malliavin
sense. We thus proceed by localization and construct a continuous process
$\bar{\Upsilon}_t$ which is smooth on the event $\{\tau_H\geq t\}$
and is close to~0 on the complementary event.
To this end we introduce the following new notations.

For some small $\eta>0$ to be fixed set
\begin{equation*}
\forall t\in[0,N],~
\bar{U}_t(v) = \{K_H^* - \textrm{Id}\}\left( 
\mathbf{1}_{[0,t]}(\sbullet)\ 
u_\lambda(B_\sbullet^H) \phi_\eta(B^H_\sbullet)\ e^{-\lambda \sbullet} 
\right)(v)
\end{equation*}
and
\begin{equation*}
\bar{\Upsilon}_t = \delta^{(N)}\left(\bar{U}_t\right),
\end{equation*}
where $\phi_\eta$ is a smooth function taking values in $[0,1]$ such 
that $\phi_\eta(x) = 1,\ \forall x\leq 1$, and $\phi_\eta(x) = 0,\ 
\forall x> 1+\eta$. 

The crucial property of $\bar{\Upsilon}_t$ is the following:
For all $n\in\N$ 
and $n\leq r\leq t<n+1$,
$\mathbf{1}_{\{\tau_H\geq t\}} 
\Upsilon_r = \mathbf{1}_{\{\tau_H\geq t\}} \bar{\Upsilon}_r$ a.s.
This is a consequence of the local property of $\delta$ 
(\cite[p.47]{Nualart}). 
Therefore, for any $n\leq N-1$,
\begin{equation}\label{eq:supYleqYbar}
\EE\left(\sup_{t\in [n,n+1]} \mathbf{1}_{\{\tau_H\geq t\}}|\Upsilon_t- 
\Upsilon_n|\right) = 
\EE\left(\sup_{t\in [n,n+1]} \mathbf{1}_{\{\tau_H\geq 
t\}}|\bar{\Upsilon}_t - \bar{\Upsilon}_n|\right) \leq 
\EE\left(\sup_{t\in [n,n+1]} |\bar{\Upsilon}_t-\bar{\Upsilon}_n|\right).
\end{equation}
Recall the Garsia-Rodemich-Rumsey lemma: if $X$ is a continuous process, then
for $p\geq 1$ and $q>0$ such that $pq>2$, one has
\begin{align}\label{eq:GRR}
\EE\left( \sup_{t\in[a,b]} |X_t-X_a| \right) &\leq C \frac{pq}{pq-2} 
(b-a)^{q-\frac{2}{p}}\ \EE\left[ \left( \int_a^b \int_a^b 
\frac{|X_s-X_t|^p}{|t-s|^{pq}}\ \dd s\ \dd t\right)^{\frac{1}{p}} 
\right] \nonumber\\
&\leq  C \frac{pq}{pq-2} (b-a)^{q-\frac{2}{p}}\ \left(\int_a^b \int_a^b 
\frac{\EE\left(|X_s-X_t|^p\right)}{|t-s|^{pq}}\ \dd s\ \dd 
t\right)^{\frac{1}{p}}
\end{align}
provided the right-hand side in each line is finite. In order to apply (\ref{eq:GRR}), we thus need to estimate moments of 
$\bar{\Upsilon}_t-\bar{\Upsilon}_s$. Note that Lemmas \ref{lem:boundAcal} and Lemmas \ref{lem:EstDeriv} (below) both give bounds on the moments of $\bar{\Upsilon}_t-\bar{\Upsilon}_s$ in terms of a power of $|t-s|$. Thus $\bar{\Upsilon}$ has a continuous modification, by Kolmogorov's continuity criterion, and the GRR lemma will be applicable to $\bar{\Upsilon}$.\\
We can easily obtain bounds on the norm 
$\left\|\bar{\Upsilon}_t-\bar{\Upsilon}_s\right\|_{L^2(\Omega)}$
in terms of $(H-\tfrac{1}{2})$. This observation
leads us to notice that
\begin{equation*}
\EE\left(|\bar{\Upsilon}_s-\bar{\Upsilon}_t|^{2+4\epsilon}\right) 
\leq 
\left\|\bar{\Upsilon}_t-\bar{\Upsilon}_s\right\|_{L^2(\Omega)}\times 
\EE\left(|\bar{\Upsilon}_t-\bar{\Upsilon}_s|^{2+8\epsilon}\right)
^{\frac{1}{2}}.
\end{equation*}
We then combine Lemmas~\ref{lem:boundAcal} and \ref{lem:EstDeriv} 
below to obtain: For every $[n\leq s\leq t\leq n+1]$,
\begin{align*}
\EE\left(|\bar{\Upsilon}_s-\bar{\Upsilon}_t|^{2+4\epsilon}\right)
&\leq  C\ (H-\tfrac{1}{2}) (t-s)^{\frac{1}{2}-\epsilon}\ 
e^{-\alpha S(1-x_0-2\eta) \sqrt{2\lambda}} \\
&\quad\quad\times 
(t-s)^{\frac{1}{2}+2\epsilon}\ e^{-\frac{1}{3} \lambda s} e^{-\alpha 
S(1-x_0-2\eta) \sqrt{2\lambda}} \\
&\leq C\ (H-\tfrac{1}{2})\ (t-s)^{1+\epsilon}\ 
e^{-\frac{1}{3}\lambda s}e^{-\alpha S(1-x_0-2\eta) \sqrt{2\lambda}} .
\end{align*}

Choosing $p= 2+4\epsilon$ and $q=\frac{2+\epsilon/2}{2+4\epsilon}$
we thus get
\begin{equation*}
\begin{split}
\EE\left(\sup_{t\in [n,n+1]} \mathbf{1}_{\{\tau_H\geq t\}}|\Upsilon_t- 
\Upsilon_n| \right) &\leq C\ 
(H-\tfrac{1}{2})^{\frac{1}{2+4\epsilon}}\ 
e^{-\frac{\alpha}{2+4\epsilon} S(1-x_0-2\eta) \sqrt{2\lambda}} \\
&\quad\quad\quad\left(\int_n^{n+1}\int_s^{n+1} e^{-\frac{1}{3}\lambda 
s} 
(t-s)^{\frac{\epsilon}{2}-1}\ \dd t \dd s 
\right)^{\frac{1}{2+4\epsilon}} \\
&\leq C\ (H-\tfrac{1}{2})^{\frac{1}{2+4\epsilon}}\ e^{-\alpha 
S(1-x_0-2\eta) \sqrt{2\lambda}} e^{-\frac{1}{3(2+4\epsilon)}\lambda n} ,
\end{split}
\end{equation*}
from which Inequality (\ref{prop:supY}) follows.
\end{proof}

It now remains to prove the above estimates on
$\left\|\bar{\Upsilon}_t-\bar{\Upsilon}_s\right\|_{L^2(\Omega)}$
and 
$\EE\left(|\bar{\Upsilon}_t-\bar{\Upsilon}_s|^{2+8\epsilon}\right)
^{\frac{1}{2}}$: These estimates are provided by 
Lemmas~\ref{lem:boundAcal} and \ref{lem:EstDeriv} 
below whose proofs are very technical.

\begin{lemma}\label{lem:boundAcal}
There exists $\eta_0\in(0,\tfrac{1-x_0}{2})$ such that: for all 
$0<\eta\leq \eta_0$, for all $H\in[\tfrac{1}{2},1)$ and for all 
$0<\epsilon<\tfrac{1}{4}$, there exist $C,\alpha>0$ such that
\begin{multline*}
\forall \lambda\geq 1,~\forall 0\leq n\leq s\leq t\leq n+1\leq N, \\
 \EE\left(|\bar{\Upsilon}_t
-\bar{\Upsilon}_s|^{2+8\epsilon}\right)^{\frac{1}{2}}
 \leq C\ (t-s)^{\frac{1}{2}+2\epsilon}\ e^{-\frac{1}{3} \lambda s} 
e^{-\alpha S(1-x_0-2\eta) \sqrt{2\lambda}}\ ,
\end{multline*}
where the function $S$ is defined as in Theorem \ref{prop:majgap_0}. 
\end{lemma}

\begin{lemma}\label{lem:EstDeriv}
There exists $\eta_0\in(0,\tfrac{1-x_0}{2})$ such that: For all 
$0<\eta\leq \eta_0$ and $0<\epsilon<\tfrac{1}{4}$, there exist 
$C,\alpha>0$ such that 
\begin{multline*}
\forall n\in [0,N],~\forall H\in[\tfrac{1}{2},1),\ \forall n\leq s\leq 
t\leq n+1,\ 
\forall\lambda\geq 1,\\
\left\|\bar{\Upsilon}_t-\bar{\Upsilon}_s\right\|_{L^2(\Omega)} \leq C\ 
(H-\tfrac{1}{2})
(t-s)^{\frac{1}{2}-\epsilon}\ e^{-\alpha S(1-x_0-2\eta) 
\sqrt{2\lambda}}.
\end{multline*}
\end{lemma}

\section{Discussion on the fBm case with $\lambda< 
1$}\label{sec:conj}

We believe that Theorem 
\ref{prop:majgap_0} also holds true for $\lambda\in(0,1]$. One of 
the main issues consists in getting accurate enough bounds on the 
right-hand side of Inequality~(\ref{eq:1stBoundI2}).

For $a_\lambda = \lambda^{-\frac{1}{2H}}$ and $b_\lambda=\tfrac{-\log 
\sqrt{\lambda}}{\lambda}$ ($\lambda<1$) we have
\begin{align*}
|I_2(\lambda)|\leq &\sqrt{2\lambda} \EE\left[\sup_{t\in [0,a_\lambda]} 
\mathbf{1}_{\{\tau_H\geq t\}} 
\left|\delta\left(\{K_H^*-\text{Id}\}(\mathbf{1}_{[0,t]} u_\lambda(B_\sbullet^H) 
e^{-\lambda\sbullet})\right)\right|\right]\\ 
&+\sqrt{2\lambda} \EE\left[\sup_{t\in [a_\lambda,b_\lambda]} 
\mathbf{1}_{\{\tau_H\geq t\}} 
\left|\delta\left(\{K_H^*-\text{Id}\}(\mathbf{1}_{[a_\lambda,t]} 
u_\lambda(B_\sbullet^H) e^{-\lambda\sbullet})\right)\right|\right] \\
&+ \sqrt{2\lambda} \lim_{N\rightarrow +\infty} \EE\left[\sup_{t\in 
[b_\lambda,N]} \mathbf{1}_{\{\tau_H\geq t\}} 
\left|\delta\left(\{K_H^*-\text{Id}\}(\mathbf{1}_{[b_\lambda,t]} 
u_\lambda(B_\sbullet^H) e^{-\lambda\sbullet})\right)\right|\right] .
\end{align*}
We here limit ourselves to examine the second summand on the r.h.s and 
we denote it by $I_2^{(2)}(\lambda)$. The two other terms (corresponding 
to $t< a_\lambda$ 
and $t>b_\lambda$) are easier to study.

Compared to Subsection \ref{subsec:proof_Thm} we localize the Skorokhod 
integral in a slightly different manner by using $\phi_\eta(S^H_t)$ 
instead of $\phi_\eta(B^H_t)$, where $S^H_t$ denotes the running 
supremum of the fBm up to time $t$. Hence
\begin{multline*}
\mathbf{1}_{\{\tau_H\geq t\}} 
\delta\left(\{K_H^*-\text{Id}\}\left(\mathbf{1}_{[0,t]} 
u_\lambda(B^H_\sbullet) e^{-\lambda \sbullet}\right)\right)\\
= \mathbf{1}_{\{\tau_H\geq t\}} 
\delta\left(\{K_H^*-\text{Id}\}\left(\mathbf{1}_{[0,t]} 
u_\lambda(B^H_\sbullet) \phi_\eta(S^H_\sbullet) e^{-\lambda \sbullet}\right)\right) \ 
\text{a.s.}
\end{multline*}
Set $\bar{V}_\lambda(s) := u_\lambda(B^H_s) \phi_\eta(S^H_s)$ and 
\begin{equation*}
\tilde{\Upsilon}_t := 
\delta\left(\{K_H^*-\text{Id}\}\left(\mathbf{1}_{[0,t]} 
\bar{V}_\lambda(\sbullet) e^{-\lambda \sbullet}\right)\right) .
\end{equation*}
Proceeding as from Eq.(\ref{eq:supYleqYbar}) to Eq.(\ref{eq:GRR}) we get
for some $p>1$ and $m>0$ (chosen 
later): 
\begin{align}\label{eq:I}
\EE \left(\sup_{t\in [a_\lambda,b_\lambda]} \mathbf{1}_{\{\tau_H\geq 
t\}} |\delta_H\left(\mathbf{1}_{[0,t]} u_\lambda(B_\sbullet^H) e^{-\lambda 
\sbullet}\right)|\right) &\leq \PP\left(\tau_H\geq 
a_\lambda\right)^{\frac{p-1}{p}} C (b_\lambda-a_\lambda)^{\frac{m}{p}} \nonumber\\
&\  \times 
\left(\int_{a_\lambda}^{b_\lambda}\int_{a_\lambda}^{b_\lambda} 
\frac{\EE\left(|\tilde{\Upsilon}_t-\tilde{\Upsilon}_s|^p\right)}{|t-s|^{m+2}}
 \ \dd s\dd t\right)^{\frac{1}{p}}\ .
\end{align}

We then use the proposition 3.2.1 in \cite{Nualart} to bound 
$\EE|\tilde{\Upsilon}_t-\tilde{\Upsilon}_s|^{p}$:
\begin{equation}\label{eq:moments}
\begin{split}
\EE|\tilde{\Upsilon}_t-\tilde{\Upsilon}_s|^p &\leq C 
(t-s)^{\frac{p}{2}-1} \int_s^t \ |\EE\left(\bar{V}_\lambda(r) e^{-\lambda 
r}\right)|^p \\
&\quad\quad + \EE\left[\left(\int_0^{b_\lambda} |D_\theta 
\bar{V}_\lambda(r) e^{-\lambda r}|^2\ \dd 
\theta\right)^{\frac{p}{2}}\right]\dd r .
\end{split}
\end{equation}
The Malliavin derivative of the supremum of the fBm is obtained for 
example in \cite{DecreusefondNualart08}. Denoting by $\vartheta_r$ 
the first time at which $B^H$ reaches $S^H_r$ on the interval $[0,r]$ 
we have $D_\theta^H S^H_r = \mathbf{1}_{\{\vartheta_r>\theta\}}$. It 
follows that $D_\theta S^H_r = K_H(\vartheta_r,\theta)$. Since 
$D_\theta \bar{V}_\lambda(r) = \phi_\eta(S^H_r) D_\theta 
u_\lambda(B^H_r) + u_\lambda(B^H_r) D_\theta \phi_\eta(S_r^H)$, we are 
led to study the three following terms (for $p>2$):
\begin{enumerate}
\item[(i)] $\EE\left(\bar{V}_\lambda(r) e^{-\lambda r}\right) \leq 
\EE\left(\phi_\eta(S^H_r)\right) \leq \PP(S^H_r\leq 1+\eta)$. 
\item[(ii)] 
$ e^{- p \lambda r}\EE\left[\left(\int_0^{b_\lambda} 
|\phi_\eta(S^H_r) D_\theta u_\lambda(B^H_r)|^2 \ \dd 
\theta\right)^{\frac{p}{2}} \right] $
\newline
$\leq \EE\left[\mathbf{1}_{\{S^H_r 
\leq 1+\eta\}} \left(\int_0^r |\sqrt{2\lambda} K_H(r,\theta) 
u_\lambda(B^H_r)|^2 \ \dd \theta\right)^{\frac{p}{2}} \right]$
\newline
$= (\sqrt{2\lambda})^{p}\ r^{pH}\ \EE(\mathbf{1}_{\{S^H_r 
\leq 1+\eta\}} u_\lambda(B^H_r)^p).$
\item[(iii)] $e^{-p \lambda r}\EE\left[\left(\int_0^{b_\lambda} 
|u_\lambda(B^H_r) D_\theta \phi_\eta(S^H_r)|^2 \ \dd 
\theta\right)^{\frac{p}{2}} \right] $
\newline
$\leq \EE\left[\phi_\eta'(S^H_r)^{p}\ \vartheta_r^{Hp} \right] 
\leq \|\phi_\eta'\|_\infty^p 
\EE\left[\mathbf{1}_{\{S_r^H\leq 
1+\eta\}} \vartheta_r^{Hp}\right].$
\end{enumerate}
We do not know any accurate estimate on the joint law of either 
$(S^H_\sbullet,B^H_\sbullet)$ or $(S^H_\sbullet,\vartheta_\sbullet)$. We thus can 
only use the rough bounds $\mathbf{1}_{\{S_r^H\leq 1+\eta\}} 
u_\lambda(B^H_r) \leq C \mathbf{1}_{\{S_r^H\leq 1+\eta\}}$ for (ii) and 
$\vartheta_r\leq r$ for (iii). Then one is in a position to use the 
following 
refinement of Molchan's asymptotic \cite{Molchan} obtained 
by Aurzada~\cite{Aurzada}: $\PP(\tau_H \geq t) \leq t^{-(1-H)} (\log t)^c$ for some constant $c>0$. 
However, when plugged into (\ref{eq:moments}) and then into~ 
(\ref{eq:I}), these bounds lead us to an upper bound for 
$|I_2^{(2)}(\lambda)|$ which diverges when $\lambda\rightarrow 0$.

Hence the preceding rough bounds on (ii) and (iii) must be improved.
In the Brownian motion case, the joint laws of $(B_r, 
S^{\frac{1}{2}}_r)$ and $(\vartheta_r,S^{\frac{1}{2}}_r)$ are known 
(see e.g. \cite[p.96--102]{KaratzasShreve}). In particular, 
for $p\in(2,3)$ the term~(iii) leads to
\begin{equation}\label{eq:boundVartheta}
\forall r\geq 0,\quad \EE\left[\mathbf{1}_{\{S^{1/2}_r\leq 1+\eta\}} 
\vartheta_r^{\frac{p}{2}}\right] \leq C
\end{equation}
instead of the bound $r^{\frac{p}{2}-\frac{1}{2}} (\log t)^c$ when one 
uses the previous rough method. 

From numerical simulations and an incomplete mathematical analysis 
using arguments developed by~\cite{Molchan} and~\cite{Aurzada} we 
believe that 
Inequality (\ref{eq:boundVartheta}) remains true for $H>\tfrac{1}{2}$. 
If so, the bound on $|I_2^{(2)}(\lambda)|$ would become
\begin{equation*}
|I_2^{(2)}(\lambda)| \leq C\sqrt{2\lambda} 
a_\lambda^{-(1-H)\frac{p-1}{p}} \ (b_\lambda-a_\lambda)^{\frac{1}{2}} ,
\end{equation*}
which, in view of $a_\lambda = \lambda^{-\frac{1}{2H}}$ and $b_\lambda 
= \frac{-\log \sqrt{\lambda}}{\lambda}$, can now be bounded as 
$\lambda\rightarrow 0$.

\section{Optimal rate of convergence in Theorem~\ref{prop:majgap_0}: 
Comparison with numerical results}\label{sec:numerics}

In this section, we numerically approximate the quantity 
$\matcal{L}(H,\lambda) = \EE\left[e^{-\lambda \tau_H}\right]$, where 
$\tau_H$ is the first time a fractional Brownian motion started from 
$0$ hits $1$.\\
As already recalled this Laplace transform  is explictely 
known in the Brownian case: $\matcal{L}(\tfrac{1}{2},\lambda) = 
e^{-\sqrt{2\lambda}}$, 
$\forall \lambda\geq 0$. Our simulations suggest that the convergence 
of $\matcal{L}(H,\lambda)$ towards $\matcal{L}(\tfrac{1}{2},\lambda)$ 
is faster than what we were able to prove. We also show numerical 
experiments which concern the convergence of hitting time densities.

Although several numerical schemes permit to decrease the weak 
error when estimating $\tau_{\frac{1}{2}}$, none seem to be available 
in the fractional Brownian motion case. We thus propose a heuristic 
extension of the 
bridge correction of Gobet \cite{Gobet} (valid in the Markov case) and 
compare this procedure to the standard Euler scheme.

\ 

\noindent \textbf{Convergence of $\EE\left[e^{-\lambda \tau_H}\right]$ 
to $\EE\big[e^{-\lambda \tau_{\frac{1}{2}}}\big]$.} 

Let us fix a time horizon $T$ and $N$ points on each trajectory. 
Let $\delta = \tfrac{T}{N}$ be the time step. 
Denote by $M$ the number of Monte-Carlo samples. 
For each $m\in\{1,\dots,M\}$, we simulate 
$\{B^{H,N}_{n\delta}(m)\}_{1\leq n\leq N}$, from which we obtain 
$\tau^{\delta,T}_H(m) = \inf\{n\delta: B^{H,N}_{n\delta}(m)>1\}$. We 
then approximate $\matcal{L}(H,\lambda)$ as follows:
\begin{equation*}
\matcal{L}(H,\lambda) \approx \frac{1}{M} \sum_{m=1}^M e^{-\lambda 
\tau^{\delta,T}_H(m)} =: \matcal{L}^{\delta,T,M}(H,\lambda)\ .
\end{equation*}
The bias $\tau^{\delta,T}_H(m) \geq \tau_H(m)$ due to the time 
discretization implies $\lim_{M\rightarrow \infty} 
\matcal{L}^{\delta,T,M}(H,\lambda)\leq \matcal{L}(H,\lambda)$.

In view of Theorem \ref{prop:majgap_0} we have
\begin{align*}
\log\left|\matcal{L}(H,\lambda) - 
\matcal{L}(\tfrac{1}{2},\lambda)\right| \leq C_\lambda + \beta\ 
\log(H-\frac{1}{2})\ ,
\end{align*}
with $\beta = (\frac{1}{4}-\epsilon)$.
We approximate $\log\left|\matcal{L}(H,\lambda) - 
\matcal{L}(\tfrac{1}{2},\lambda)\right|$ by 
$\log\left|\matcal{L}^{\delta,T,M}(H,\lambda) - 
\matcal{L}(\tfrac{1}{2},\lambda)\right|$  for several values of $H$ 
close to $\tfrac{1}{2}$ and then perform a linear regression analysis 
around $\log(H-\tfrac{1}{2})$. The slope of the regression line 
provides a hint on the optimal value of $\beta$.

Notice that the global error $|\matcal{L}(H,1) - 
\matcal{L}^{\delta,T,M}(H,1)|$ results from the discretization error
$\text{error}(\delta)$ and the statistical error~$\text{error}(M)$.
The chosen number of simulations $M=10^5$ is such that 
$|\text{error}(M)|\leq C/\sqrt{M} \approx 7.10^{-4}$, for some numerical constant $C>0$.

The numerical results are presented in Table \ref{tab:3} for several 
values of $\lambda (=1,2,3,4)$ and of the parameter $H\in\{0,5; 0,51; 
0,52; 0,54; 0,6\}$. 
These results suggest that 
$|\matcal{L}^{\delta,T,M}(\tfrac{1}{2},\lambda) - 
\matcal{L}^{\delta,T,M}(H,\lambda)|$ is linear w.r.t.
$(H-\tfrac{1}{2})$. For each $\lambda$ we thus perform a linear 
regression on these quantities (without the above $\log$ 
transformation). The regression line is plotted in Fig.~\ref{fig:Reg}.

\begin{figure}[!h]
\centering
\includegraphics[scale=.5]{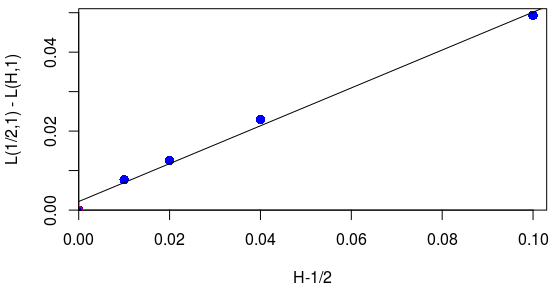}
\caption{Regression of $\matcal{L}(\tfrac{1}{2},1) - \matcal{L}(H,1)$ against $H-\tfrac{1}{2}$ using the values from Table \ref{tab:3}.}
\label{fig:Reg}
\end{figure}

Our numerical results suggest that Theorem~\ref{prop:majgap_0} is not 
optimal but the optimal convergence rate seems hard to get. An even 
more difficult result to obtain concerns the convergence 
rate of the density of the first 
hitting time of fBm to the density of the first hitting time of 
Brownian motion. We analyze it numerically: See 
Fig.~\ref{fig:Densities}.

\begin{figure}[!h]
\centering
\includegraphics[scale=.55]{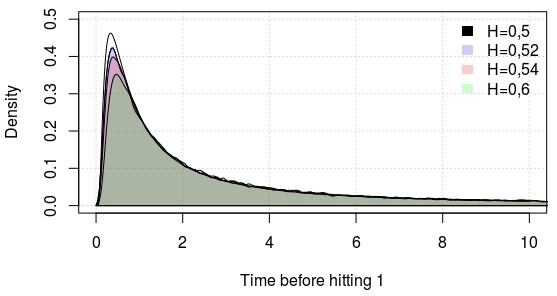}
\caption{Density of $\tau_H$ for several values of $H$}
\label{fig:Densities}
\end{figure}

\ 

\noindent \textbf{Brownian bridge correction.} We apply the following 
rule (which is only heuristic when $H>\tfrac{1}{2}$):
at each time step, if the threshold has not yet been hit and if 
$B^{H,N}_{(n-1)\delta}(m)<1$ and $B^{H,N}_{n\delta}(m)<1$, we sample a 
uniform random variable $U$ on $[0,1]$ and compare it to 
\begin{equation*}
p_H = 
\exp\left\{-2\frac{\left(1-B^{H,N}_{(n-1)\delta}(m)\right)
\left(1-B^{H,N}_{n\delta}(m)\right)}{\delta^{2H}}\right\}. 
\end{equation*}
If $U<p_H$ then decide $\tau^{\delta,T}_H(m) = n\delta$. 
Otherwise let the algorithm continue. We denote by 
$\widetilde{\matcal{L}}^{\delta,T,M}(H,\lambda)$ the corresponding 
Laplace transform. This algorithm is an adaptation to a non-Markovian framework of the algorithm of \cite{Gobet}, which is rigorously proven when $H=\tfrac{1}{2}$. In particular $p_{\frac{1}{2}}$ is exactly the probability that a Brownian motion conditioned by its values at time $(n-1)\delta$ and $n\delta$ crosses $1$ in the time interval $[(n-1)\delta,n\delta]$.

Table \ref{tab:1} shows the corresponding 
results for the simple estimator 
$\matcal{L}^{\delta_0,T,M}(\tfrac{1}{2},\lambda)$ and the Brownian 
Bridge
estimator $\widetilde{\matcal{L}}^{\delta_1,T,M}(\tfrac{1}{2},\lambda)$ 
with $\delta_0<\delta_1$ in the Brownian case (we kept $M=10^5$). 
Consistently with theoretical 
results, Table 
\ref{tab:1} shows that the estimator 
$\widetilde{\matcal{L}}^{\delta,T,M}(H,\lambda)$ allows to 
substantially reduce the number of discretization steps (thus the
computational time) to get a desired accuracy. The figure also shows 
a reasonable choice of $\delta_1$ which we actually keep 
when tackling the fractional Brownian motion case.

The exact value $\matcal{L}(H,\lambda)$ is unknown. Our reference value 
is the lower bound
$\matcal{L}^{\delta_0,T,M}(H,\lambda)$.
The parameter $\delta_1$ used in Table \ref{tab:2_0} allows to 
conjecture that the Brownian bridge correction is useful even in the 
non-Markovian case. 
Although the approximation errors of the estimators 
$\matcal{L}^{\delta_1,T,M}$ and 
$\widetilde{\matcal{L}}^{\delta_1,T,M}$ are similar when compared
to $\matcal{L}^{\delta_0,T,M}(H,\lambda)$, we recommend to use the 
latter because we have
$\matcal{L}^{\delta_1,T,M}(H,\lambda)\leq 
\matcal{L}^{\delta_0,T,M}(H,\lambda)\leq \matcal{L}(H,\lambda)$ whereas 
$\matcal{L}^{\delta_0,T,M}(H,\lambda)\leq 
\widetilde{\matcal{L}}^{\delta_1,T,M}(H,\lambda)$.


\section*{Appendix: tables}

\begin{table}[!h]
\small
\begin{center}
Set of parameters: $T=20,\ N=2^{16}\ (\delta\approx 3.10^{-4}),\ M=10^5$
\end{center}
\vspace{-0.8cm}
\begin{center}
\begin{tabular}{p{0.8cm}|p{1.5cm}p{1cm}|p{1.5cm}p{1cm}|p{1.5cm}p{1cm}|p{1.5cm}p{1cm}}
\noalign{\smallskip}
    & $\lambda=1$ &   & $\lambda=2$ &   & $\lambda=3$ &   & $\lambda=4$ 
    &  \\
\hline
$H$ & $\matcal{L}^{\delta,T,M}(H,\lambda)$ & \hspace{0.4cm}$\Delta_H$ & 
$\matcal{L}^{\delta,T,M}(H,\lambda)$ & \hspace{0.4cm}$\Delta_H$ & 
$\matcal{L}^{\delta,T,M}(H,\lambda)$ & \hspace{0.4cm}$\Delta_H$ & 
$\matcal{L}^{\delta,T,M}(H,\lambda)$ & \hspace{0.4cm}$\Delta_H$ \\
\hline\noalign{\smallskip}\noalign{\smallskip}
$0,50$ & $0,2400$  & \hspace{0.4cm}-- & $0,1329$  & \hspace{0.4cm}-- & 
$0,0846$  & \hspace{0.4cm}-- & $0,0578$  & \hspace{0.4cm}-- \\
\hline\noalign{\smallskip}
$0,51$ & $0,2323$  & $0,0077$ & $0,1271$  & $0,0059$ & $0,0800$  & 
$0,0046$ & $0,0542$  & $0,0037$ \\
\hline\noalign{\smallskip}
$0,52$ & $0,2275$  & $0,0125$ & $0,1232$  & $0,0098$ & $0,0769$  & 
$0,0077$ & $0,0517$  & $0,0061$ \\
\hline\noalign{\smallskip}
$0,54$ & $0,2171$  & $0,0229$ & $0,1149$  & $0,0180$ & $0,0703$  & 
$0,0143$ & $0,0464$  & $0,0114$ \\
\hline\noalign{\smallskip}
$0,60$ & $0,1907$  & $0,0493$ & $0,0958$  & $0,0372$ & $0,0560$  & 
$0,0286$ & $0,0354$  & $0,0224$ \\
\noalign{\smallskip}\noalign{\smallskip}
\end{tabular}
\end{center}
\vspace{-0.3cm}
\caption{Values of $\Delta_H = \EE\big[e^{-\lambda 
\tau_{\frac{1}{2}}}\big] - \EE\left[e^{-\lambda \tau_H}\right]$ when 
$H\rightarrow\tfrac{1}{2}$.}
\label{tab:3}
\end{table}

\begin{table}[!ht]
\small 
\begin{center}
Set of parameters: $T=20,\ N=2^{16}\ (\delta_0\approx 3.10^{-4}),\ 
M=10^5$ for the simple estimator\\
\hspace{3cm} $T=20,\ N=2^{15}\ (\delta_1\approx 6.10^{-4}),\ M=10^5$ 
for the Bridge estimator
\end{center}
\vspace{-0.8cm}
\begin{center}
\begin{tabular}{p{1.5cm}p{2cm}p{2cm}p{1.5cm}p{2cm}p{1.5cm}}
\hline\noalign{\smallskip}
$\lambda$ & $\matcal{L}(\tfrac{1}{2},\lambda)$ & 
$\matcal{L}^{\delta,T,M}(\tfrac{1}{2},\lambda)$ & Error (\%) & 
$\widetilde{\matcal{L}}^{\delta,T,M}(\tfrac{1}{2},\lambda)$ & Error 
(\%) \\
\hline\noalign{\smallskip}\noalign{\smallskip}
$1$ & $0,2431$  & $0,2400$ & $1,3$ & $0,2438$ & $0,3$\\
$2$ & $0,1353$  & $0,1329$ & $1,7$ & $0,1358$ & $0,4$\\
$3$ & $0,0863$ & $0,0846$ & $2,0$ & $0,0867$ & $0,5$\\
$4$ & $0,0591$  & $0,0578$ & $2,2$ & $0,0594$ & $0,5$\\
\noalign{\smallskip}\hline\noalign{\smallskip}
\end{tabular}
\end{center}
\vspace{-0.3cm}
\caption{Test case: Error estimation of our procedure in the Brownian 
case ($H=\tfrac{1}{2}$)}
\label{tab:1}   
\end{table}

\begin{table}[!ht]   
\small
\begin{center}
Set of parameters: $T=20,\ N=2^{16}\ (\delta_0\approx 1,5.10^{-4}),\ 
M=10^5$ for the simple estimator\\
\hspace{2.6cm} $T=20,\ N=2^{15}\ (\delta_1\approx 6.10^{-4}),\ M=10^5$ 
for the simple estimator\\
\hspace{2.6cm} $T=20,\ N=2^{15}\ (\delta_1\approx 6.10^{-4}),\ M=10^5$ 
for the Bridge estimator
\end{center}
\vspace{-0.8cm}
\begin{center}
\begin{tabular}{p{1.8cm}p{2cm}p{2cm}p{1.5cm}p{2cm}p{1.5cm}}
\hline\noalign{\smallskip}
$\lambda$ & $\matcal{L}^{\delta_0,T,M}(H,\lambda)$ & 
$\matcal{L}^{\delta_1,T,M}(H,\lambda)$ & Error (\%) & 
$\widetilde{\matcal{L}}^{\delta_1,T,M}(H,\lambda)$ & Error (\%) \\
\hline\noalign{\smallskip}\noalign{\smallskip}
$1$ & $0,2171$  & $0,2147$ & $1,1$ & $0,2186$ & $0,7$\\
$2$ & $0,1149$ & $0,1131$ & $1,6$ & $0,1165$ & $1,4$\\
$3$ & $0,07003$ & $0,0689$ & $2,0$ & $0,0717$ & $1,9$\\
$4$ & $0,0464$  & $0,0453$ & $2,3$ & $0,0476$ & $2,5$\\
\noalign{\smallskip}\hline\noalign{\smallskip}
\end{tabular}
\end{center}
\vspace{-0.3cm}
\caption{Comparison of estimators in the fractional case ($H=0,54$)}
\label{tab:2_0} 
\end{table}

\end{document}